\documentclass[10pt]{article}
\usepackage{macros}
\usepackage{xurl}

\bibliography{references}

\usepackage{subfiles}

\title{Sharp lower bounds for the periodic maximal Schr\"odinger operator in higher dimensions}
\author{Inbo Gottlieb Fenves and Jia Hao Tan}
\date{}

\begin{document}

\maketitle

\begin{abstract}
    We prove sharp lower bounds for the Schr\"odinger maximal function in $\T^d$ for all dimensions $d \ge 2$, and as a corollary obtain sharp regularity conditions for pointwise convergence of the periodic Schr\"odinger equation. Combined with sufficiency results established by Compaan-Luc\'a-Staffilani \cite{CompaanLucaStaffilani2021}, this yields a full resolution of Carleson's problem up to endpoint for the periodic Schr\"odinger equation in all dimensions $d \ge 2$, and disproves the conjectured regularity condition of Miao-Yuan-Zhao \cite{MiaYuaZhaBar}, in contrast to the corresponding question in $\R^d$. We also prove sharp estimates on the dimensions of divergence sets for the equation for high dimensional tori $d \ge 2$. Our approach uses complex multiplication on abelian varieties.
\end{abstract}

\section{Introduction}

The primary goal of this paper is to prove the following theorem.

\begin{thm}\label{thm:frequencies}
    Fix $d \ge 2$. For each $R \in \N$ there exists an $R^{\frac{d}{d+2}}$-separated $\Sigma \subseteq \{-R,...,R\}^d$ with $|\Sigma| \asymp R^{\frac{2d}{d+2}}$ and an $R^{-1}$-separated set $X \subseteq \T^d$ with $|X| \asymp R^d$ so that for all $x \in \Nc_{c_dR^{-1}}(X)$ we have
    \begin{equation*}
        \sup_{0 \le t \le 1}\left| \sum_{\xi \in \Sigma} e \bigl(x \cdot \xi + t|\xi|^2 \bigr) \right| \gtrsim_d R^{\frac{d}{d+2}}|\Sigma|^{1/2}.
    \end{equation*}
    Moreover, the relevant times $t \in [0,1]$ are evenly spaced at scale $R^{-\frac{2d}{d+2}}$, and contribute equal mass.
\end{thm}

As an immediate corollary, we obtain the following lower bound on $(\ell^2,L^q)$-bounds for the periodic maximal Schr\"odinger operator.

\begin{thm}\label{thm:schrodingerlowerbd}
    For all $d \ge 2$ and all $R \in \N$ there exists a sequence $b \in \ell^2$ for which
    \begin{equation*}
        \norm{\sup_{0 \le t \le 1} \biggl| \sum_{\substack{\nbf \in \Z^d \\ |n_i| \le R}} b_\nbf e\bigl(x \cdot \nbf + t|\nbf|^2\bigr) \biggr| }_{L_x^q(\T^d)} \gtrsim_{q,d} R^{\frac{d}{d+2}}\norm{b}_{\ell^2}
    \end{equation*}
    holds simultaneously for $q \in [1,\infty]$.
\end{thm}

By a classical application of Stein's maximal principle, Theorem \ref{thm:schrodingerlowerbd} in turn implies the following necessary regularity condition for a.e. convergence of the Schr\"odinger operator.

\begin{cor}\label{cor:sobolevregularity}
    Let $d \ge 2$, and suppose $s < \frac{d}{d+2}$. Then there exists $f \in H^s(\T^d)$ and a set $E \subset \T^d$ with $|E| > 0$ so that for any $x \in E$,
    \begin{align*}
        \limsup_{t \to 0^{+}} \vert e^{it\Delta} f(x) \vert = \infty.
    \end{align*}
\end{cor}

The key distinction accounting for the discrepancy in regularity is an additional recurrence present in the periodic case. Note that, if $g \in \Mat_{d \times d}(\Z)$ has nonzero determinant, then the action of $g$ on $\R^d$ distorts the measure by a factor of $\det(g)$; on the other hand, $g$ descends to a well-defined endomorphism of $\T^d = \R^d/\Z^d$, and there $g$ preserves the Lebesgue measure. This will allow us to pull back a set of large peaks inside of a dilated copy of the torus while preserving measure estimates.

We are also able to obtain a fractal analogue of Corollary \ref{cor:sobolevregularity}. To state, we define the \textit{divergence set} of a function $f \colon \T^d \to \C$ by
\begin{equation*}
    \Dc(f) = \left\{ x \in \T^d : \text{$e^{it\Delta}f(x)$ does not converge to $f(x)$ as $t \to 0$}  \right\},
\end{equation*}
and for $s \ge 0$, define the \textit{divergence dimension} $\alpha_{\T^d}(s)$ by
\begin{equation*}
    \alpha_{\T^d}(s) = \sup_{f \in H^s(\T^d)} \dim_H \Dc(f).
\end{equation*}
Eceizabarrena-Luca showed \cite{Eceizabarrena2022} that
\begin{align}
    \alpha_{\T^d}(s) &= d,  &s < \frac{d}{2(d+1)} \label{eq:EFd} \\
    \alpha_{\T^d}(s) &\ge d+1-\frac{2(d+1)}{d}s, &  s \in \left[ \frac{d}{2(d+1)},\frac{d}{2} \right] \label{eq:ELlowerbd} \\
    \alpha_{\T^d}(s) &\le d+2-\frac{2(d+2)}{d}s, &s \in \left[ \frac{d}{d+2},\frac{d}{2} \right], \label{eq:ELupperbd}
\end{align}
and the following bound was conjectured in analogy with the Euclidean case:
\begin{equation}\label{eq:ELconj}
    \alpha_{\T^d}(s) = d+1 - \frac{2(d+1)}{d}s, \quad \quad s \in \left[ \frac{d}{2(d+1)},\frac{d}{2} \right].
\end{equation}
(For the case of $\alpha_{\R^d}(s)$, see \cite{LucaRogers2019}, \cite{DuGuthLi2017}, \cite{DuZhang2019}, \cite{DuKimWangZhang2020}, \cite{LucPon} and the references therein). Note that Corollary \ref{cor:sobolevregularity} disproves the conjectured \eqref{eq:ELconj} in the range $s \in \bigl( \frac{d}{2(d+1)}, \frac{d}{d+2}\bigr)$, and that $\Dc(f) = \emptyset$ when $s > d/2$ by the Sobolev embedding theorems. A priori, it is unclear for a general $f \in H^s(\T^d)$ whether or not $e^{it\Delta}f(x)$ is even well-defined $\Hc^\alpha$-a.e., where $\Hc^\alpha$ is the $\alpha$-dimensional Hausdorff measure, but it was shown in \cite{Eceizabarrena2022} that this expression is in fact $\Hc^\alpha$-a.e. well-defined (where by convention $e^{it\Delta}f(x)$ is the pointwise limit of its cubic partial Fourier sums) provided that $s > \frac{d-\alpha}{2}$, and therefore Theorem \ref{thm:fractalregularity} is well-posed in this regime.

We are able to show that \eqref{eq:ELconj} fails for \textit{every} $s \in \bigl( \frac{d}{2(d+1)},\frac{d}{2} \bigr)$, and in fact that the upper bound of \eqref{eq:ELupperbd} is sharp in its stated range, using a fractal analogue of Theorem \ref{thm:frequencies} with an essentially identical proof.

\begin{thm}\label{thm:fractalregularity}
    Suppose $d \ge 2$, $0 < \alpha \le d$, and set $\gamma_d(\alpha) = \frac{d}{2(d+2)}(d+2-\alpha)$. Then, if $0 \le s < \gamma_d(\alpha)$, there exists $f \in H^s(\T^d)$ and a Borel subset $E \subseteq \T^d$ with $\dim_H E \ge \alpha$ for which
    \begin{equation*}
        \limsup_{t \to 0} |e^{it\Delta}f(x)| = \infty
    \end{equation*}
    for every $x \in E$.
\end{thm}

\textbf{A brief history of the problem}. Corollary \ref{cor:sobolevregularity} settles Carleson's problem on the critical Sobolev exponent for almost everywhere pointwise convergence of the periodic Schr\"odinger equation for all dimensions $d \ge 2$ up to the endpoint, in view of the positive results of Compaan-Luc\`a-Staffilani \cite[Proposition 3.1]{CompaanLucaStaffilani2021}, who in turn utilized Strichartz estimates established by Bourgain-Demeter \cite{BourgainDemeter2015}. Prior work on the maximal Weyl sum estimates, alongside the corresponding Euclidean problem, suggested that Sobolev regularity $s > \frac{d}{2(d+1)}$ was sufficient (See Moyua-Vega \cite{MoyVeg}, Barron \cite{Barron2022}, Baker \cite{Bak}, and Demeter \cite{Demeter2024} for $\T^1$, Miao-Yuan-Zhao-Barron \cite{MiaYuaZhaBar} and  Eceizabarrena-Yu \cite{EceYu} for $\T^d$, and Fenves-Tan \cite{FenTan} for perturbative results), whereas Corollary \ref{cor:sobolevregularity} shows that the conjectured behavior does not hold for $d \ge 2$. The case $d=1$ remains open; for further discussion, see Section \ref{sec:dim1}. This is in striking contrast to the Euclidean problem, where the case $d = 1$ was resolved early (including the endpoint) by Carleson \cite{Carleson1980} and Dahlberg-Kenig \cite{DahKen}, who showed $s \ge \frac{1}{4}$ is both necessary and sufficient for $\R^1$. Long after the resolution of the $d= 1$ case, Bourgain \cite{Bourgain2016} proved that $s \ge \frac{d}{2(d+1)}$ was necessary. Sufficiency for $s > \frac{d}{2(d+1)}$ was subsequently shown by Du-Guth-Li \cite{DuGuthLi2017} and Du-Zhang \cite{DuZhang2019}; \cite{DuGuthLi2017} for $d=2$ through an $L^3$-Schr\"odinger maximal estimate, and \cite{DuZhang2019} for all higher dimensions $d \geq 3$ via a fractal $L^2$-restriction argument. For prior work on divergence set estimates in the Euclidean setting, we direct the reader toward Lucà and Rogers \cite{LucaRogers2019}, and Du, Kim, Wang, and Zhang \cite{DuKimWangZhang2020}, and references therein. 

\textbf{Notation}. We use the standard notation $e(s) = e^{2\pi i s}$, and normalize the Schr\"odinger operator by
\begin{equation*}
    e^{it\Delta}f(x) = \sum_{\xi \in \Z^d} \hat{f}(\xi)e(x \cdot \xi + t|\xi|^2) \quad \mathrm{for} \quad f(x) = \sum_{\xi \in \Z^d} \hat{f}(\xi)e(x \cdot \xi).
\end{equation*}
Given a set $X \subseteq \T^d$ and $r > 0$, we let $\Nc_r(X) = \{x \in \T^d : \dist(x,X) \le r\}$ for the \textit{$r$-neighborhood} of $X$. The notation $A \lesssim B$ or $A = O(B)$ means $A \le CB$ for some constant $C > 0$, and $A \asymp B$ if $A \lesssim B$ and $B \lesssim A$. The notation $A \ll B$ means $A \le cB$ for some $c > 0$ which is assumed to be sufficiently small. Throughout, quantities will depend on a scale $R \ge 1$; the notation $A \lessapprox B$ means $A \lesssim_\epsilon R^\epsilon B$ for all $\epsilon > 0$. For auxiliary parameter(s) $par$, we use subscripts in asymptotic notation to denote dependency; for example, $A \lesssim_{par} B$ if $A \le C_{par} B$ for some $C_{par} > 0$ depending on $par$. Our implicit constants are allowed to implicitly depend on the dimension $d$, but are required to be independent of $R$.

\textbf{Acknowledgments}. The first author was supported by NSF RTG grant No. DMS-2230900. The second author was supported by the National Science Foundation under grant No. DMS-2037851. 

After a first draft of this manuscript was completed, an alternative proof of Theorem \ref{thm:schrodingerlowerbd} was independently obtained by Cen-Gao-Zhang \cite{CenGaoZha} using affine congruence classes and counting estimates in finite abelian groups. Their argument more closely parallels the initial disproof generated during the authors' September 17th, 2026 conversation with ChatGPT. For a detailed AI disclosure, see Section \ref{sec:AI}.

\section{Proof of Theorem \ref{thm:frequencies}}

In this section, we prove Theorem \ref{thm:frequencies}, and as a consequence Theorem \ref{thm:schrodingerlowerbd} and Corollary \ref{cor:sobolevregularity}. In the even-dimensional case, we identify $\T^{2g} = \C^g/\Lambda$ with a CM abelian variety; under this identification, the corresponding set of frequencies $\Sigma$ is a well-chosen coset of a subgroup $\Bar{\varpi}^k\Lambda$ of $\Lambda$ coming from the $\Z[i]$-action. In the odd-dimensional case, we introduce a twist by a character in the remaining variable.

\subsection{Dimension two}

To start, we consider the case $d=2$, where most of the key ideas are already present. The reason to restrict to the case $d=2$ is that $\T^2 = \C/\Z[i]$ is a CM elliptic curve. We will view the Gaussian integers simultaneously as a lattice in $\C$ and as an algebra of endomorphisms; to emphasize this distinction, we write $\Lambda$ when referring to $\Z[i]$ as a lattice. Given $a \in \Z[i]$, we write $M_a \in \End(\Lambda)$ for the endomorphism given by the multiplication action $x \mapsto ax$.

Recall that an odd rational prime $p \in \Z$ splits in $\Z[i]$ into $p = \varpi \Bar{\varpi}$ for some Gaussian prime $\varpi \in \Z[i]$ if and only if $p \equiv 1 \mod{4}$. For $d=2$, we may specialize to the case $p = 5$, $\varpi = 2+i$, $\Bar{\varpi} = 2-i$. Choose some $k \in \N$ and let $N = 5^{k}$. Note that we have a factorization $M_N = M_{\varpi^{k}}M_{\Bar{\varpi}^{k}}$, with $\det M_{\varpi^{k}} = \det M_{\Bar{\varpi}^{k}} = N$. Moreover $M_N$ is the zero endomorphism on the quotient $\Lambda/N\Lambda$. We have an exact sequence
\begin{equation}\label{eq:exactsequence}
    \begin{tikzcd}
        0 \ar[r] & \Bar{\varpi}^{k}\Lambda/N\Lambda \ar[r] & \Lambda/N\Lambda \ar[r, "{M_{\varpi^{k}}}"] & \varpi^{k}\Lambda/N\Lambda \ar[r] & 0
    \end{tikzcd}
\end{equation}
The subgroup $G = \im [M_{\varpi^{k}}]$ is a cyclic subgroup of order $N$ by unique factorization in $\Z[i]$.

Define a norm $\norm{\cdot}$ on $\Lambda$ given by the standard absolute value. We can then define an associated periodic metric $\norm{\cdot}_{N}$ by
\begin{equation*}
    \norm{\alpha}_{N} = \min_{\beta \in N\Lambda}\norm{\alpha-\beta}.
\end{equation*}

\begin{lemma}\label{lemma:seplemmad=2}
    Suppose $\alpha \in G$ is nonzero. Then $\norm{\alpha}_{N} \ge N^{1/2}$.
\end{lemma}
\begin{proof}
    Write $\alpha = j\varpi^{k}$. Then we have
    \begin{equation*}
        \norm{\alpha}_{N} = \min_{\delta \in \Lambda} |\varpi^{k}(j- \Bar{\varpi}^{k}\delta)| = N^{1/2} \min_{\delta \in \Lambda}|j - \Bar{\varpi}^{k}\delta|.
    \end{equation*}
    Note that $j - \Bar{\varpi}^{k}\delta \in \Lambda$ is a nonzero Gaussian integer, and hence $|j - \Bar{\varpi}^{k}\delta| \ge 1$.
\end{proof}

\begin{proof}[Proof of Theorem \ref{thm:schrodingerlowerbd} for $d=2$]
    Since $N^{-1/2}M_{\varpi^{k}} \in \Unitary(1)$ is an isometry for the norm $\norm{\cdot}$, we can always find a set $A \subseteq \Lambda$ so that $\Bar{\varpi}^{k}A$ projects to the full subgroup $\Bar{\varpi}^{k}\Lambda/N\Lambda$ modulo $N$, and with
    \begin{equation}\label{eq:Asizebd}
        \max_{a \in A} \norm{a} < 2N^{1/2}.
    \end{equation}
    We then set $\Sigma = \varpi^k + \Bar{\varpi}^kA$.

    Note that we have the identity
    \begin{equation*}
        |\varpi^{k} + \Bar{\varpi}^{k}a|^2 = N(1+|a|^2) + 2\Re(\Bar{\varpi}^{2k}a).
    \end{equation*}
    Now, let us define $u \colon \C/\Lambda \times [0,1] \to \C$ by
    \begin{equation*}
        u(x,t) = \sum_{\xi \in \Sigma} e\bigl( x \cdot \xi + t|\xi|^2 \bigr),
    \end{equation*}
    where $x \cdot \xi = x_1\xi_1 + x_2\xi_2 = \Re(\Bar{x}\xi)$ for $\xi = \xi_1 + i\xi_2$. In particular, for each $0 \le j < N$, we have
    \begin{align*}
        u(x,jN^{-1}) &= e\bigl( x \cdot \varpi^{k} \bigr)\sum_{a \in A} e\bigl(x \cdot \Bar{\varpi}^{k}a + 2j\Re(\Bar{\varpi}^{2k}a)/N) \\
        &= e(x \cdot \varpi^{k} )\sum_{a \in A} e\bigl( a \cdot \bigl[ \varpi^{k}x + 2\varpi^{2k} j/N \bigr] \bigr).
    \end{align*}
    Since $N$ is odd and $\varpi^{k}$ is a generator of $G$, so is $\varpi^{2k}$. Moreover, since $2$ is invertible in $G$, the points $z_j = -2\varpi^{2k} j/N$ are distinct for $1 \le j \le N$. By Lemma \ref{lemma:seplemmad=2}, we have $|z_j - z_{j'}| \ge N^{-1/2}$ for each $j \ne j'$.

    Let $X = M_{\varpi^k}^{-1}\{z_1,...,z_N\}$; then $X$ is $N^{-1}$-separated. Moreover, if $x \in \Nc_{cN^{-1}}(X)$, then there exists some $w \in \C$ with $|w| \lesssim N^{-1/2}$ for which $\varpi^{k}x - z_j \equiv w \mod{\Lambda}$ for some $j$. Since $|a \cdot w| \le |a||w| \lesssim 1$ by \eqref{eq:Asizebd}, the phase $e\left( a \cdot[\varpi^{k}x + 2\varpi^{2k} j/N] \right) = e\left( a \cdot w \right)$ has definite real part (perhaps up to shrinking the neighborhoods by an $O_d(1)$-factor), for each $a \in A$. Thus, we obtain
    \begin{equation*}
        \norm{\sup_{0 \le t \le 1} |u(x,t)|}_{L_x^p(\C/\Lambda)} \gtrsim_p |A|.
    \end{equation*}
    Since $|A| = |\Sigma| = N$, we thus have
    \begin{equation*}
        \norm{\sup_{0 \le t \le 1} |u(x,t)|}_{L_x^q(\C/\Lambda)} \gtrsim_p N,
    \end{equation*}
    and taking $R = N$, we are done.
\end{proof}

\subsection{Even dimensions}

In this section, we extend Theorem \ref{thm:schrodingerlowerbd} to hold for all even $d \ge 2$.

To start, choose a Galois extension $K/\Q(i)$ of degree $g$, and let $p \in \Z$ be an odd rational prime which splits completely in $K$. Let $p = \varpi \Bar{\varpi}$ with $\varpi,\Bar{\varpi} \in \Z[i]$ prime. Choose a basis $\vbf_1,...,\vbf_g \in \Os_K$ realizing $\Lambda = \Os_K$ as a free rank-$g$ $\Z[i]$-module, and endow $\Lambda$ with the induced complex inner product
\begin{equation}
    \inner{\sum_j a_j\vbf_j}{\sum_j b_j\vbf_j} = \sum_j a_j\Bar{b}_j.
\end{equation}
We let $\norm{a} = \sqrt{\inner{a}{a}}$ be the corresponding norm. Note that $\Z[i]$ once again acts by similitudes on $\Lambda$ with respect to this metric. We will use the following observation.

\begin{lemma}
    For all $\alpha \in K$, we have
    \begin{equation}\label{eq:fieldnormtonorm}
        |\Norm_\Q^K(\alpha)| \lesssim_K \norm{\alpha}^{2g}.
    \end{equation}
\end{lemma}
\begin{proof}
    Note that $|\Norm_\Q^K(\alpha)|$ may be expressed as the product of absolute values of $\alpha$ under the $2g$ embeddings $K \hookrightarrow \C$, and each such embedding yields an absolute value dominated by $\norm{\alpha}$.
\end{proof}

Now, fix $k \in \N$, and let $N = p^{k}$. Then we have a decomposition
\begin{equation*}
    \Lambda/N\Lambda \cong \prod_{\pfrak \mid p} \Lambda/\pfrak^{k} \cong (\Zmod{N})^{2g}.
\end{equation*}
Let $\norm{\cdot}_N$ be the periodic metric on $\Lambda/N\Lambda$ induced by $\norm{\cdot}$ on $\Lambda$.

Considering the endomorphism $M_{\varpi^{k}}$ acting on $\Lambda/N\Lambda$, we have
\begin{equation}\label{eq:varpikerim}
    \ker[M_{\varpi^{k}}] \cong \prod_{\pfrak \mid \varpi} \Lambda/\pfrak^k \quad \mathrm{and} \quad \im[M_{\varpi^{k}}] \cong \prod_{\Bar{\pfrak} \mid \Bar{\varpi}} \Lambda/\Bar{\pfrak}^k.
\end{equation}

\begin{lemma}\label{lemma:beta}
    There exists $\beta \in \im [M_{\varpi^k}]$ of order $N$, so that for each $\alpha \in \gen{\beta}$ nonzero, we have
    \begin{equation}\label{eq:periodicnormlowerbd}
        \norm{\alpha}_N \gtrsim_K N^{1-1/d}.
    \end{equation}
\end{lemma}
\begin{proof}
    Fix any $\Bar{\qfrak} \mid \Bar{\varpi}$, and choose $\beta$ whose image under the isomorphism of \eqref{eq:varpikerim} is 1 in the $\Bar{\qfrak}$-component, and zero otherwise. Given any $\alpha \in \gen{\beta}$, lift to some $\Tilde{\alpha} \in \Lambda$, which is necessarily nonzero. Then $\Tilde{\alpha}$ is a nonzero element of the ideal $\afrak = N\Bar{\qfrak}^{-k}\Os_K$, and thus
    \begin{equation}\label{eq:fieldnormtoidealnorm}
        |\Norm_\Q^K(\Tilde{\alpha})| \ge [\Os_K:\afrak] = N^{d-1}.
    \end{equation}
    Combining \eqref{eq:fieldnormtoidealnorm} with \eqref{eq:fieldnormtonorm} yields the desired result.
\end{proof}

\begin{proof}[Proof of Theorem \ref{thm:schrodingerlowerbd} for $d = 2g$]
    Once again, we have that $N^{-1/2}M_{\varpi^{k}} \in \Unitary(g)$ is an isometry for $\norm{\cdot}$. Thus, we may choose some $\gamma \in \Lambda$ with $\norm{\gamma} \lesssim N^{1/2}$ and $\varpi^k\gamma = \beta$ with $\beta$ as in \ref{lemma:beta}. Let us set $R = N^{1/2 + 1/d}$, and take
    \begin{equation*}
        \Sigma = \gamma + \Bar{\varpi}^k\Lambda \cap B_R^d,
    \end{equation*}
    where $B_R^d$ is taken with respect to the metric $\norm{\cdot}$. Note that $|\Sigma| \asymp N$. We may choose $A \subseteq \Lambda$ with $\max_{a \in A}\norm{a} \lesssim N^{1/d}$ so that $\Sigma = \gamma + \Bar{\varpi}^kA$. The identity
    \begin{equation*}
        \norm{\gamma + \Bar{\varpi}^ka}^2 = \norm{\gamma}^2 + 2\Re{\inner{\gamma}{\Bar{\varpi}^ka}} + \norm{\Bar{\varpi}^ka}^2 = N\norm{a}^2 + \norm{\gamma}^2 + 2\Re{\inner{\beta}{a}}
    \end{equation*}
    allows us to cancel all the relevant quadratic terms. Thus, if we define
    \begin{equation*}
        u(x,t) = \sum_{\xi \in \Sigma} e\bigl( \Re{\inner{x}{\xi}} + t\norm{\xi}^2 \bigr),
    \end{equation*}
    then for all $1 \le j < N$ we have
    \begin{align*}
        u(x,jN^{-1}) &= e\bigl( \Re{\inner{x}{\gamma}} + j\norm{\gamma}^2N^{-1} \bigr) \sum_{a \in A}e \bigl( \Re{\inner{x}{\Bar{\varpi}^ka}} + 2j \Re{\inner{\beta}{a}}/N \bigr) \\
        &= e\bigl( \Re{\inner{x}{\gamma}}+j\norm{\gamma}^2N^{-1} \bigr) \sum_{a \in A}e \bigl( \Re{\inner{a}{\varpi^kx + 2j\beta/N}} \bigr).
    \end{align*}
    If we set $z_j = -2j\beta/N$, then Lemma \ref{lemma:beta} guarantees that $\norm{z_j - z_{j'}} \gtrsim N^{-1/d}$ for distinct indices $1 \le j,j' \le N$. If we let $X = M_{\varpi^k}^{-1}\{z_1,...,z_N\}$, then $|X| \gtrsim N^{\frac{d+2}{2}} \asymp R^d$, and moreover for all $x \in \Nc_{c_dR^{-1}}(X)$ we see that $e( \Re{\inner{a}{\varpi^kx -z_j}} )$ has real part $\gtrsim 1$, so that
    \begin{equation*}
        \sum_{1 \le j \le N} \int_{\Nc_{c_dR^{-1}}(M_{\varpi^{k}}^{-1}\{z_j\})} |u(x,jN^{-1})|^q\dx \gtrsim |\Nc_{c_dR^{-1}}(X)||A|^q \gtrsim N^q = R^{q \frac{2d}{d+2}},
    \end{equation*}
    or
    \begin{equation*}
        \norm{\sup_{0 \le t \le 1} |u(x,t)|}_{L^q(\T^d)} \gtrsim_q R^{\frac{2d}{d+2}} = R^{\frac{d}{d+2}}|\Sigma|^{1/2},
    \end{equation*}
    as desired.
\end{proof}

\subsection{Odd dimensions}

In odd dimensions, we employ a parity trick to obtain slightly larger separation. Let $d = 2g+1$, and consider $\T^{2g+1} = \C^g/\Lambda \times \T^1$. Our goal is to construct the frequency set $\Sigma$ as a graph over the $2g$-dimensional one; the additional $\T^1$-factor will allow us to produce some additional separation in the frequencies.

Let us fix $K,\Lambda,p,\varpi,\beta$ as in the even-dimensional case, and assume that $k$ is divisible by $2d$ (so that all quantities involved are integers). We extend the norm $\norm{\cdot}$ to $\Lambda \times \Z$ via $\norm{(\alpha,j)}^2 = \norm{\alpha}^2 + j^2$. Consider the submodule $N \Lambda \times N \Z$, and let $\norm{\cdot}_N$ be the associated periodic metric on $\Lambda/N\Lambda \times \Zmod{N}$. Let $G = \gen{\beta}$, and let $\chi \in \hat{G}$ be the character $\chi(j\beta) = jN^{1-1/d} \mod{N}$. We then set
\begin{equation*}
    H = \left\{ (j\beta, \chi(j\beta)) \in \Lambda/N\Lambda \times \Zmod{N} : j \in \Zmod{N} \right\}.
\end{equation*}

\begin{lemma}
    For all $1 \le j < N$, we have
    \begin{equation*}
        \norm{(j\beta,\chi(j\beta))}_N \gtrsim N^{1-1/d}.
    \end{equation*}
\end{lemma}
\begin{proof}
    Note that, if $N^{1/d} \nmid j$, then the second coordinate is of size at least $N^{1-1/d}$, as desired. If $N^{1/d} \mid j$, then (in the notation of Lemma \ref{lemma:beta}), we note that $j\beta$ is a nonzero element in the ideal $\bfrak = N\Bar{\qfrak}^{-k+k/d} \Os_K$, so that
    \begin{equation*}
        |\Norm_\Q^K(j\beta)| \ge [\Os_K:\bfrak] = N^{d-2+1/d}.
    \end{equation*}
    This, alongside \eqref{eq:fieldnormtonorm}, yields the desired result.
\end{proof}

\begin{proof}[Proof of Theorem \ref{thm:schrodingerlowerbd} for $d=2g+1$]
    We consider $M_{(\varpi^{k},N^{1/2})}$ acting on $\C^g/\Lambda \times \T^1$. Once again, $N^{-1/2}M_{(\varpi^{k},N^{1/2})} \in \Ortho(d)$ is an isometry of $\norm{\cdot}$. Choose $\gamma \in \Lambda$ and $h \in \Zmod{N}$ so that $\norm{\gamma} \lesssim N^{1/2}$, $|h| \lesssim N^{1/2}$, and with $M_{(\varpi^k,N^{1/2})}(\gamma,h) = (\beta,\chi(\beta))$. Again set $R = N^{1/2 + 1/d}$, and take
    \begin{equation*}
        \Sigma = (\gamma,h) + \Bar{\varpi}^k\Lambda \times N^{1/2}\Z \cap B_R^d.
    \end{equation*}
    Then $|\Sigma| \asymp N$. From here, the argument is analogous to the even-dimensional case. Since $N \in p^{2d\N}$ is arbitrary, we are done.
\end{proof}

\begin{rmk}
    While it was convenient for us to work with powers of a  fixed totally splitting prime $p$, one may alternatively fix $k = 1$ (or, for the odd-dimensional case, $k=2d$), and let $p$ vary amongst all such totally split primes, which have relative density $\gtrsim_K 1$ by Chebotarev density.
\end{rmk}

\section{Divergence on fractals}

In this section, we prove Theorem \ref{thm:fractalregularity}. We make use of the following fractal analogue of Theorem \ref{thm:frequencies}, whose proof follows exactly as that of Theorem \ref{thm:frequencies}, with suitable choices of radii.

\begin{prop}\label{prop:frequenciesfractal}
    Fix $d \ge 2$, and for each $0 < \alpha \le d$, set $\gamma_d(\alpha) = \frac{d}{2(d+2)}(d+2-\alpha)$. Then for all $R \in \N$ there exist an $R^{\frac{\alpha}{d+2}}$-separated set $\Sigma \subseteq \{-R,...,R\}^d$ with $|\Sigma| \asymp_{\alpha} R^{2\gamma_d(\alpha)}$ and an $R^{-\alpha/d}$-separated subset $X \subseteq \T^d$ with $|X| \asymp_{\alpha} R^{\alpha}$ so that for all $x \in \Nc_{c_\alpha R^{-1}}(X)$ we have
    \begin{equation*}
        \sup_{0 \le t \le 1} \left| \sum_{\xi \in \Sigma} e \bigl( x \cdot \xi + t|\xi|^2 \bigr) \right| \gtrsim_{d,\alpha} R^{\gamma_d(\alpha)}|\Sigma|^{1/2}
    \end{equation*}
    Moreover, the relevant times $t \in [0,1]$ are evenly spaced at scale $R^{-\frac{2\alpha}{d+2}}$, and contribute equal mass.
\end{prop}
\begin{proof}
    For sake of notation we consider the case $d=2g$ even; the odd case is exactly the same. In the proof of Theorem \ref{thm:frequencies}, take $R \asymp N^{\frac{d+2}{2\alpha}}$; then we have
    \begin{equation*}
        |\Sigma| \asymp R^dN^{-d/2} \asymp R^{2\gamma_d(\alpha)}.
    \end{equation*}
    We may choose a set $A \subseteq \Lambda$ with
    \begin{equation}\label{eq:Asizefractal}
        \max_{a \in A} \norm{a} \lesssim RN^{-1/2} \lesssim R^{\frac{d+2-\alpha}{d+2}}
    \end{equation}
    so that $\Sigma = \gamma + \Bar{\varpi}^kA$. 
    
    Again, choose $z_1,...,z_N \in \T^d$ as in the proof of Theorem \ref{thm:frequencies}, and take $X = M_{\varpi^k}^{-1}\{z_1,...,z_N\}$. Note that points in $X$ are separated by at least $N^{-1/2-1/d} \asymp R^{-\alpha/d}$, and $|X| \gtrsim N^{\frac{d+2}{2}} \asymp R^\alpha$. Moreover, for any $x \in \Nc_{c_dR^{-1}}(X)$, we have by \eqref{eq:Asizefractal} that 
    \begin{equation*}
        \sup_{0 \le t \le 1}\left| \sum_{\xi \in \Sigma} e(x \cdot \xi + t|\xi|^2) \right| \gtrsim_{d,\alpha} R^{2\gamma_d(\alpha)}.
    \end{equation*}
\end{proof}

Since the case $\alpha = d$ follows from Corollary \ref{cor:sobolevregularity}, we will now assume $0 < \alpha < d$; we assume that $\frac{d-\alpha}{2} < s < \gamma_d(\alpha)$, so that the problem is well-posed by \cite{Eceizabarrena2022}. 

Our goal is to construct an exceptional set $K \subseteq \T^d$ as a nested intersection of sets which are fractal at various scales using Proposition \ref{prop:frequenciesfractal} alongside parabolic rescaling. To begin the induction, set $q_0,R_0,\delta_0,m_0 = 1$, and let $K_0$ be an embedded ball in $\T^d$ of radius $r_0 \gtrsim 1$. We also fix some large constant $C \gg_\alpha 1$.

Now, assume we have chosen $q_{n-1},R_{n-1},r_{n-1},\delta_{n-1},m_{n-1},K_{n-1}$, so that $K_{n-1}$ is a union of $m_0 \cdots m_{n-1}$ balls of radius $r_{n-1}$, whose centers are $\delta_{n-1}$-separated. Set $q_{n} = \floor{Cr_{n-1}^{-1}}$, $R_{n} \gg R_{n-1}$, and
\begin{equation*}
    \Tilde{X}_{R_n} = \{x \in \T^d : q_nx \in X_{R_n}\},
\end{equation*}
where $X_{R_n}$ is obtained as in Proposition \ref{prop:frequenciesfractal}. We then set $r_{n} \asymp (q_{n}R_{n})^{-1}$ and $\delta_{n} \asymp (q_{n}R_{n}^{\alpha/d})^{-1}$. Then $\Tilde{X}_{R_n}$ consists of $\gtrsim q_n^dR_n^\alpha$ points which are $\delta_n$-separated. For each $x \in X_{R_n}$, the set $Z_x = \{z \in \T^d : q_nz = x\}$ is a grid of spacing $q_n^{-1}$; thus, we have the lower bound
\begin{equation}\label{eq:ballintersection}
    |B_{r_{n-1}/2}(y) \cap Z_x| \gtrsim q_n^{d}r_{n-1}^d
\end{equation}
independent of $y \in \T^d$, where the implicit constants are independent of $C$, so long as it is sufficiently large. In particular, if $B \subseteq K_{n-1}$ is an $r_{n-1}$-ball, then
\begin{equation*}
    \biggl|\frac{1}{2}B \cap \Tilde{X}_{R_n}\biggr| \gtrsim q_n^dR_n^\alpha r_{n-1}^d,
\end{equation*}
where $\frac{1}{2}B$ is the ball with the same center and half the radius. We then define $K_n \subseteq K_{n-1}$ choosing exactly $m_n$ balls of radius $r_n$ lying inside each of the parent balls of $K_{n-1}$. If we set $m_n \asymp q_n^dR_n^\alpha r_{n-1}^d$, then $K_n$ consists of $m_0 \cdots m_n$ balls of radius $r_n$, with centers $\delta_n$-separated. Finally, we set $K = \bigcap_n K_n$.

\begin{lemma}
    The set $K$ satisfies $\dim_H K \ge \alpha$.
\end{lemma}
\begin{proof}
    Define a probability measure $\mu$ on $K$ assigning equal mass to each $r_n$-ball in $K_n$. Each level-$n$ ball has mass $w_n = (m_0 \cdots m_n)^{-1}$. We claim that we may choose $C$ sufficiently large so that
    \begin{equation}\label{eq:masstoradius}
        w_n \lesssim r_n^\alpha
    \end{equation}
    Indeed, we observe that
    \begin{equation}
        m_n \frac{r_n^\alpha}{r_{n-1}^\alpha} \asymp q_n^dR_n^\alpha r_n^\alpha r_{n-1}^{d-\alpha} \asymp (q_nr_{n-1})^{d-\alpha} \asymp C^{d-\alpha},
    \end{equation}
    with the implicit constant independent of the choice of $C \gg 1$. Since we assumed $\alpha < d$, it follows that for $C = C_0$ sufficiently large in relation to all other parameters, we have that
    \begin{equation*}
        r_n^\alpha > r_{n-1}^\alpha m_n^{-1},
    \end{equation*}
    after which \eqref{eq:masstoradius} follows by induction.

    Now, the lemma follows from the mass distribution principle. Indeed, $m_n \asymp \delta_n^{-d}r_{n-1}^d$ by definition. If we have $r_n < r \le r_{n-1}$, then since $K_n$ consists of $\delta_n$-separated balls, any $r$-ball in $\T^d$ intersects $K_n$ in $\lesssim 1 + (r/\delta_n)^d$ balls. Thus, for $x \in \T^d$ arbitrary we have
    \begin{equation*}
        \mu(B_r(x)) \lesssim w_n\left(1 + (r/\delta_n)^d \right) \lesssim r_n^\alpha + w_{n-1}\frac{r^d}{r_{n-1}^d} \lesssim r_n^\alpha + r^\alpha \left( \frac{r}{r_{n-1}} \right)^{d-\alpha} \lesssim r^\alpha.
    \end{equation*}
    In particular, this implies $\dim_H \mu \ge \alpha$, and since $\supp \mu = K$ it follows that $\dim_H K \ge \alpha$.
\end{proof}

We are now in place to prove Theorem \ref{thm:fractalregularity}.

\begin{proof}[Proof of Theorem \ref{thm:fractalregularity}]
    Fix some $s < \gamma_d(\alpha)$. For each $n$, let $P_n,\phi_n$ be the trigonometric polynomials produced by Proposition \ref{prop:frequenciesfractal}, alongside the appropriate rescalings:
    \begin{equation*}
        P_n(x) = \sum_{\xi \in \Sigma_{R_n}} e\bigl( x \cdot \xi \bigr), \quad \quad \phi_n(x) = R_n^{-2\gamma_d(\alpha)}P_{n}(q_nx).
    \end{equation*}
    It is straightforward to check that
    \begin{equation*}
        \norm{\phi_n}_{H^s} \lesssim (q_nR_n)^s \norm{\phi_n}_{L^2} \asymp q_n^sR_n^{s-\gamma_d(\alpha)}.
    \end{equation*}
    In particular, if we choose $(R_n)_{n \in \N}$ to grow sufficiently fast that $q_n^sR_n^{s-\gamma_d(\alpha)} \lesssim 2^{-n}$, then we have that $\phi_n \to 0$ in $H^s(\T^d)$. On the other hand, by parabolic rescaling we see that for every $x \in K$ there exists $t_x \in (0,q_n^{-2}]$ for which
    \begin{equation*}
        |e^{it\Delta}\phi_n(x)| \ge 1.
    \end{equation*}
    By the Banach principle (see e.g. \cite[Theorem 1.3]{Lit}) combined with \cite[Proposition A.1]{Eceizabarrena2022}, there exists some $f \in H^s(\T^d)$ for which
    \begin{equation*}
        \limsup_{t \to 0} |e^{it\Delta}f(x)| = \infty
    \end{equation*}
    for a subset of $K$ of positive $\mu$-measure; this completes the proof.
\end{proof}

\section{The case of dimension 1}\label{sec:dim1}

The analogue of Theorem \ref{thm:frequencies} for $d=1$ would be the following (unestablished) statement.

\begin{statement}\label{statement:frequenciesd=1}
    For all $R \in \N$, there exists an $R^{1/3}$-separated set of frequencies $\Sigma \subseteq \{-R,...,R\}$ with $|\Sigma| \asymp R^{2/3}$ and an $R^{-1}$-separated set $X \subseteq \T^1$ with $|X| \asymp R$ so that for all $x \in \Nc_{cR^{-1}}(X)$ we have
    \begin{equation*}
        \sup_{0 \le t \le 1}\left| \sum_{\xi \in \Sigma} e \bigl(x \cdot \xi + t|\xi|^2 \bigr) \right| \gtrsim R^\frac{2}{3}.
    \end{equation*}
\end{statement}

We contrast Statement \ref{statement:frequenciesd=1} with recent work of Demeter \cite{Demeter2024} on $L^4$-level set estimates for the maximal Schr\"odinger function on $\T^1$.

\begin{thm}[\cite{Demeter2024}, Theorem 1.1]\label{thm:Dem24}
    Suppose $R \gg 1$ and $R^{7/20} \lessapprox \lambda \lesssim R^{1/2}$. Then for any sequence $(b_n)_{n=1}^R$ of complex coefficients, we have
    \begin{equation*}
        \left|\left\{x \in \T^1 : \sup_{0 \le t \le 1}\biggl|\sum_{n=1}^R b_ne\bigl( x n + tn^2 \bigr)\biggr| \ge \norm{b}_{\ell^2}\lambda \right\} \right| \lessapprox \frac{R}{\lambda^4}.
    \end{equation*}
\end{thm}

For Statement \ref{statement:frequenciesd=1}, the relevant level sets are at height $\lambda \asymp R^{1/3}$, and since $R^{1/3} \ll R^{7/20}$, the regimes of \cite{Demeter2024} do not apply. It was conjectured in \cite[Conjecture 1.2]{Demeter2024} that Theorem \ref{thm:Dem24} could extend all the way to $R^{1/4} \leq \lambda \leq R^{1/2}$, which would establish the sufficiency of $s > \frac{1}{4}$ in $d = 1$. If one could extend Theorem \ref{thm:Dem24} to the range $\lambda \in [R^{\frac{1}{3}-c}, R^\frac{7}{20}]$, then Statement \ref{statement:frequenciesd=1} would be refuted; on the other hand, if Statement \ref{statement:frequenciesd=1} holds, then $s \ge \frac{1}{3}$ is necessary for pointwise convergence on $\T^1$, and \cite[Conjecture 1.2]{Demeter2024} would fail whenever $\lambda \asymp R^{1/3}$.

Note that the proof of Theorem \ref{thm:frequencies} completely breaks down in the case of dimension 1. The reason for this is that $\T^1$ rather trivially does not admit complex multiplication. More specifically, in the proof of Theorem \ref{thm:frequencies}, we locate a cyclic subgroup $G \le (\Zmod{N})^d$ of size $N$ whose nonzero elements all stay sufficiently far from the origin. When $d \ge 2$ this construction can be realized by using the ``diagonal directions'' inside of this non-cyclic subgroup, or equivalently by observing that $1 \in \Os_K/\pfrak$ is a generator in a ``twisted'' way, when $\pfrak \mid p$ is a split prime above some $p \in \Z$. On the other hand, when $d = 1$ it is obvious that no such ``geometrically nontrivial'' subgroup $G \le \Zmod{N}$ is available.

A plausible strategy towards trying to obtain counterexamples to \cite[Conjecture 1.2]{Demeter2024} would be via a projection argument. Take some initial datum in $f \in H^s(\T^d)$ given by
\begin{equation*}
    f(x) = \sum_{1 \le \xi \le R} a_{\xi} e(\xi x),
\end{equation*}
and lift to $\T^2$ by $F(x_1,x_2) = f(x_1)$ so $\hat{F}(n,m) = 1_{\{m=0\}}a_n$; the one-dimensional frequencies of $f$ are then supported along a horizontal line $\ell$ in two dimensions. Theorem \ref{thm:frequencies} at $d=2$ would give an $R^{1/2}$-separated set of frequencies $\Sigma \subseteq \{1,...,R\}^2$ with $|\Sigma| \asymp R$ and an $R^{-1}$-separated set $X \subseteq \T^2$ with $|X| \asymp R^2$ so that for all $x \in \Nc_{cR^{-1}}(X)$,
\begin{equation*}
    \sup_{0 \le t \le 1}\left| \sum_{\xi \in \Sigma} e \bigl(x \cdot \xi + t|\xi|^2 \bigr) \right| \gtrsim R.
\end{equation*}
Due to the $R^{1/2}$-separation condition on the frequencies $\Sigma$, a given horizontal line $\ell$ can contain at most $R^{1/2}$ of the frequencies. In particular, even a line $\ell$ were to saturate the packing, we could expect at best an estimate of the form
\begin{align}
    \sup_{0 \le t \le 1}\left| \sum_{\xi \in \Sigma \cap \ell} e \bigl(x \cdot \xi + t|\xi|^2 \bigr) \right| \asymp R^{1/2}.
\end{align}
This would only give the necessity of $s \geq \frac{1}{4}$, and therefore provide no additional gain over the examples of Bourgain.

\section{AI Disclosure}\label{sec:AI}

The original proofs of Theorems \ref{thm:schrodingerlowerbd} and \ref{thm:fractalregularity} were found by ChatGPT 6 Astra on September 17, 2026, in a discussion on previous results by the authors; the proofs are available on the first author's website \href{https://sites.google.com/wisc.edu/inbo/research?authuser=0}{here}. The AI proof of Theorem \ref{thm:frequencies} used factorization over number fields in the form of Fermat's two squares theorem, but was otherwise essentially combinatorial, involving certain counting estimates for subgroups of $(\Zmod{N})^d$ for $N$ a large prime (or square of a large prime). The authors subsequently realized that the proof could be significantly simplified and made deterministic by interpreting the relevant objects in the language of CM abelian varieties; the resulting argument is what is presented here. The proof of Theorem \ref{thm:fractalregularity} was simplified from its original form, but with no essential new ideas introduced. All writing is human-generated.

\nocite{*}
\printbibliography

@article{Eceizabarrena2022,
  author  = {Eceizabarrena, Daniel and Luc{\`a}, Renato},
  title   = {Convergence over fractals for the periodic {Schr{\"o}dinger} equation},
  journal = {Analysis \& PDE},
  year    = {2022},
  volume  = {15},
  number  = {7},
  pages   = {1775--1805},
  doi     = {10.2140/apde.2022.15.1775}
}

@article{CompaanLucaStaffilani2021,
  author  = {Compaan, Erin and Luc{\`a}, Renato and Staffilani, Gigliola},
  title   = {Pointwise convergence of the {S}chr{\"o}dinger flow},
  journal = {International Mathematics Research Notices},
  year    = {2021},
  volume  = {2021},
  number  = {1},
  pages   = {596--650},
  doi     = {10.1093/imrn/rnaa036}
}

@article{Bourgain2016,
  author  = {Bourgain, J.},
  title   = {A note on the Schrödinger maximal function},
  journal = {Journal d'Analyse Mathématique},
  year    = {2016},
  volume  = {130},
  pages   = {393--396},
  doi     = {10.1007/s11854-016-0042-8}
}

@article{BourgainDemeter2015,
  author  = {Bourgain, Jean and Demeter, Ciprian},
  title   = {The proof of the {$l^2$} decoupling conjecture},
  journal = {Annals of Mathematics},
  year    = {2015},
  volume  = {182},
  number  = {1},
  pages   = {351--389},
  doi     = {10.4007/annals.2015.182.1.9}
}

@inproceedings{Carleson1980,
  author    = {Carleson, Lennart},
  title     = {Some analytic problems related to statistical mechanics},
  booktitle = {Euclidean Harmonic Analysis},
  editor    = {Burdock, J. J.},
  series    = {Lecture Notes in Mathematics},
  volume    = {779},
  pages     = {5--45},
  publisher = {Springer},
  address   = {Berlin, Heidelberg},
  year      = {1980},
  doi       = {10.1007/BFb0087666}
}

@article{DuZhang2019,
  author  = {Du, Xiumin and Zhang, Ruixiang},
  title   = {Sharp {$L^2$} estimates of the {Schr{\"o}dinger} maximal function in higher dimensions},
  journal = {Annals of Mathematics},
  year    = {2019},
  volume  = {189},
  number  = {3},
  pages   = {837--861},
  doi     = {10.4007/annals.2019.189.3.4}
}

@article{DuGuthLi2017,
  author  = {Du, Xiumin and Guth, Larry and Li, Xiaochun},
  title   = {A sharp {Schr{\"o}dinger} maximal estimate in {$\mathbb{R}^2$}},
  journal = {Annals of Mathematics},
  year    = {2017},
  volume  = {186},
  number  = {2},
  pages   = {607--640},
  doi     = {10.4007/annals.2017.186.2.5}
}

@article{LucaRogers2019,
  author  = {Luc{\`a}, Renato and Rogers, Keith M.},
  title   = {A note on pointwise convergence for the {Schr{\"o}dinger} equation},
  journal = {Mathematical Proceedings of the Cambridge Philosophical Society},
  year    = {2019},
  volume  = {166},
  number  = {2},
  pages   = {209--218},
  doi     = {10.1017/S0305004117000743}
}

@article{DuKimWangZhang2020,
  author  = {Du, Xiumin and Kim, Jongchon and Wang, Hong and Zhang, Ruixiang},
  title   = {Lower bounds for estimates of the {Schr{\"o}dinger} maximal function},
  journal = {Mathematical Research Letters},
  year    = {2020},
  volume  = {27},
  number  = {3},
  pages   = {687--692},
  doi     = {10.4310/MRL.2020.v27.n3.a4}
}

@article{Demeter2024,
  author  = {Demeter, Ciprian},
  title   = {Level set estimates for the periodic {S}chr{\"o}dinger maximal function on {$\mathbb{T}^1$}},
  journal = {Advances in Mathematics},
  year    = {2024},
  doi     = {10.48550/arXiv.2402.01099},
  note    = {To appear}
}

@article{MiaYuaZhaBar, title={Maximal estimates for Weyl sums on T d}, volume={284}, ISSN={00221236}, url={https://linkinghub.elsevier.com/retrieve/pii/S0022123622003676}, DOI={10.1016/j.jfa.2022.109747}, number={2}, journal={Journal of Functional Analysis}, author={Miao, Changxing and Yuan, Jiye and Zhao, Tengfei and Barron, Alex}, year={2023}, month=jan, pages={109747}, language={en} }

@article{Lit, title={On continuity at zero of the maximal operator for a semifinite measure}, volume={135}, ISSN={0010-1354, 1730-6302}, url={http://journals.impan.pl/cgi-bin/doi?cm135-1-6}, DOI={10.4064/cm135-1-6}, number={1}, journal={Colloquium Mathematicum}, author={Litvinov, Semyon}, year={2014}, pages={79–84}, language={en} }

@article{LucPon, title={Convergence over fractals for the Schroedinger equation}, volume={71}, ISSN={0022-2518}, url={http://www.iumj.indiana.edu/IUMJ/fulltext.php?artid=9302&year=2022&volume=71}, DOI={10.1512/iumj.2022.71.9302}, number={6}, journal={Indiana University Mathematics Journal}, author={Luca, R. and Ponce-Vanegas, Felipe}, year={2022}, pages={2283–2307}, language={en} }

@misc{FenTan, title={Maximal estimates for perturbations of the Schr\"odinger operator on $\mathbb{T}^d$}, url={https://arxiv.org/abs/2608.07464v1}, journal={arXiv.org}, author={Fenves, Inbo Gottlieb and Tan, Jiahao}, year={2026}, month=aug, language={en} }

@article{CenGaoZha, title={A sharp regularity threshold for Schrödinger maximal estimates on standard tori}, url={http://arxiv.org/abs/2609.29047}, DOI={10.48550/arXiv.2609.29047}, note={arXiv:2609.29047 [math.AP]}, number={arXiv:2609.29047}, publisher={arXiv}, author={Cen, Xi and Gao, Xitao and Zhang, Junyong}, year={2026}, month=sept }

@article{Barron2022,
  author  = {Barron, Alex},
  title   = {An {$L^4$} maximal estimate for quadratic {W}eyl sums},
  journal = {International Mathematics Research Notices},
  year    = {2022},
  volume  = {2022},
  number  = {22},
  pages   = {17822--17849},
  doi     = {10.1093/imrn/rnab181}
}

@incollection{DahKen,
  author    = {Dahlberg, B. E. J. and Kenig, C. E.},
  title     = {A note on the almost everywhere behaviour of solutions to the {Schr{\"o}dinger} equation},
  booktitle = {Harmonic Analysis: Proceedings of a Conference Held at the University of Minnesota, Minneapolis, April 20--30, 1981},
  series    = {Lecture Notes in Mathematics},
  volume    = {908},
  pages     = {205--209},
  year      = {1982},
  publisher = {Springer},
  address   = {Berlin, Heidelberg}
}

@article{BarBenCarRog, title={On the dimension of divergence sets of dispersive equations}, volume={349}, rights={http://www.springer.com/tdm}, ISSN={0025-5831, 1432-1807}, url={http://link.springer.com/10.1007/s00208-010-0529-z}, DOI={10.1007/s00208-010-0529-z}, number={3}, journal={Mathematische Annalen}, author={Barceló, Juan Antonio and Bennett, Jonathan and Carbery, Anthony and Rogers, Keith M.}, year={2011}, month=mar, pages={599–622}, language={en} }

@article{MoyVeg, title={Bounds for the maximal function associated to periodic solutions of one-dimensional dispersive equations}, volume={40}, rights={http://doi.wiley.com/10.1002/tdm_license_1.1}, ISSN={00246093}, url={http://doi.wiley.com/10.1112/blms/bdm096}, DOI={10.1112/blms/bdm096}, number={1}, journal={Bulletin of the London Mathematical Society}, author={Moyua, A. and Vega, L.}, year={2008}, month=feb, pages={117–128}, language={en} }

@article{Bak, title={$L^p$ maximal estimates for quadratic Weyl sums}, volume={165}, ISSN={0236-5294, 1588-2632}, url={https://link.springer.com/10.1007/s10474-021-01173-3}, DOI={10.1007/s10474-021-01173-3}, number={2}, journal={Acta Mathematica Hungarica}, author={Baker, R.}, year={2021}, month=dec, pages={316–325}, language={en} }

@article{EceYu, title={Uniform periodic counterexamples to Carleson’s convergence problem with polynomial symbols}, url={http://arxiv.org/abs/2408.13935}, DOI={10.48550/arXiv.2408.13935}, note={arXiv:2408.13935 [math.AP]}, number={arXiv:2408.13935}, publisher={arXiv}, author={Eceizabarrena, Daniel and Yu, Xueying}, year={2025}, month=sept }

@article{MocTao, title={Restriction and Kakeya phenomena for finite fields}, volume={121}, ISSN={0012-7094}, url={https://projecteuclid.org/journals/duke-mathematical-journal/volume-121/issue-1/Restriction-and-Kakeya-phenomena-for-finite-fields/10.1215/S0012-7094-04-12112-8.full}, DOI={10.1215/S0012-7094-04-12112-8}, number={1}, journal={Duke Mathematical Journal}, author={Mockenhaupt, Gerd and Tao, Terence}, year={2004}, month=jan }

\vspace{\fill}
\noindent
\textsc{Department of Mathematics, University of Wisconsin--Madison, 480 Lincoln Dr., Madison, WI 53706, USA}\\[2pt]
\textit{Email address:} \texttt{gottliebfenv@wisc.edu}\\
\textit{Email address:} \texttt{jtan84@wisc.edu}
\end{document}